\documentclass[11pt,a4paper]{article}

\usepackage[T1]{fontenc}
\usepackage[utf8]{inputenc}
\usepackage{lmodern}
\usepackage{microtype}

\usepackage[a4paper,margin=1in]{geometry}
\usepackage{amsmath,amssymb,amsthm,mathtools,mathrsfs,stmaryrd,bm,bbm}

\usepackage{tikz}
\usepackage{tikz-cd}
\usetikzlibrary{arrows.meta,positioning,calc,fit,backgrounds,shapes.geometric}
\tikzset{>=Latex}

\usepackage{booktabs,array,longtable,tabularx,ltablex}
\keepXColumns

\usepackage{enumitem}
\usepackage{xcolor}
\definecolor{linkcol}{RGB}{22,90,160}
\usepackage[
  colorlinks=true,
  linkcolor=linkcol,
  citecolor=linkcol,
  urlcolor=linkcol
]{hyperref}
\usepackage[nameinlink,noabbrev]{cleveref}

\usepackage{caption}
\usepackage{subcaption}

\newcommand{\Ga}{\Gamma}

\newcommand{\Ocal}{\mathcal{O}}

\newcommand{\Cat}[1]{\mathbf{#1}}

\newcommand{\QCoh}{\Cat{QCoh}}
\newcommand{\Perf}{\Cat{Perf}}
\newcommand{\Mot}{\Cat{Mot}}

\newcommand{\Spec}{\operatorname{Spec}}
\newcommand{\SpecGnC}[1]{\mathrm{Spec}^{\mathrm{nc}}_{\Ga}\!\bigl(#1\bigr)}
\newcommand{\nTGMod}[1]{#1\text{-}\Ga\mathrm{Mod}}

\newcommand{\Ch}{\mathbf{Ch}}

\newcommand{\triplearrows}{\rightrightarrows\!\!\rightrightarrows}

\DeclareMathOperator{\Hom}{Hom}

\DeclareMathOperator{\Proj}{Proj}

\newcommand{\Fcal}{\mathcal{F}}
\newcommand{\Gcal}{\mathcal{G}}
\newcommand{\Ecal}{\mathcal{E}}
\newcommand{\Ccal}{\mathcal{C}}

\newcommand{\Kspec}{\mathbb{K}}

\DeclareMathOperator*{\holim}{holim}
\DeclareMathOperator{\Map}{Map}
\DeclareMathOperator{\Fun}{Fun}
\DeclareMathOperator{\HC}{HC}
\DeclareMathOperator{\TC}{TC}
\newcommand{\perf}{\mathrm{perf}}

\newcommand{\AffnGamma}{\Cat{Aff}^{\mathrm{nc}}_{\Ga}}

\theoremstyle{plain}
\newtheorem{theorem}{Theorem}[section]
\newtheorem{lemma}[theorem]{Lemma}
\newtheorem{proposition}[theorem]{Proposition}
\newtheorem{corollary}[theorem]{Corollary}

\theoremstyle{definition}
\newtheorem{definition}[theorem]{Definition}

\newtheorem{notation}[theorem]{Notation}
\newtheorem{setup}[theorem]{Setup}

\theoremstyle{remark}
\newtheorem{remark}[theorem]{Remark}

\crefname{theorem}{Theorem}{Theorems}
\crefname{lemma}{Lemma}{Lemmas}
\crefname{proposition}{Proposition}{Propositions}
\crefname{corollary}{Corollary}{Corollaries}
\crefname{definition}{Definition}{Definitions}
\crefname{example}{Example}{Examples}
\crefname{remark}{Remark}{Remarks}
\crefname{setup}{Setup}{Setups}

\title{FUNDAMENTAL COMPARISON, BASE-CHANGE, AND DESCENT THEOREMS IN THE $K$-THEORY OF NON-COMMUTATIVE $n$-ARY $\Gamma$-SEMIRINGS}

\author{Chandrasekhar Gokavarapu\\
Lecturerin Mathematics  Mathematics, Government College (Autonomous), \\
Rajahmundry,A. P.,India PIN: 533105\\
Department of Mathematics, Acharya Nagarjuna University,  Guntur, 522510, A. P.,India\\chandrasekhargokavarapu@gmail.com}

\begin{document}

\maketitle

\begin{abstract}
We develop a comparison, base--change, and descent framework for the algebraic
$K$--theory of non--commutative $n$--ary $\Gamma$--semirings. Working in the
Quillen--exact (and Waldhausen) setting of bi--finite, slot--sensitive
$\Gamma$--modules and perfect complexes, we construct functorial maps on
$K$--theory induced by extension and restriction of scalars under explicit
$\Gamma$--flatness hypotheses in the relevant positional slots. We prove
derived Morita invariance (via tilting bi--module complexes), establish
Beck--Chevalley type base--change for cartesian squares, and deduce a
projection formula compatible with the multiplicative structure coming from
positional tensor products. Passing to the non--commutative
$\Gamma$--spectrum $\Spec^{\mathrm{nc}}_\Gamma(T)$, we show locality for
perfect objects and derive Zariski hyperdescent for $\Kspec(\Perf)$, together
with excision and localization sequences for closed immersions and fpqc
descent for $\Gamma$--flat covers. Finally, we interpret $K_\Gamma(X)$
geometrically as the $K$--theory of the stable $\infty$--category of
$\Gamma$--perfect complexes, describe its universal property in
$\Gamma$--linear non--commutative motives, and record compatibility with
cyclotomic and Chern--type trace maps.
\end{abstract}

\medskip
\noindent\textbf{Keywords:}
non--commutative $n$--ary $\Gamma$--semiring; positional tensor product;
exact category; Waldhausen category; Quillen $Q$--construction;
$S_\bullet$--construction; perfect $\Gamma$--complex; derived Morita invariance;
base--change; projection formula; Zariski descent; fpqc descent; excision;
localization; non--commutative motives; cyclotomic trace.

\medskip
\noindent\textbf{Mathematics Subject Classification (2020):}
Primary: 19D10, 19E08, 16Y60.
Secondary: 18E30, 55U35.

\section{Introduction}

Algebraic $K$--theory is a universal receptacle for the ``linear''
information carried by an algebraic object \cite{Weibel2013}: finitely generated projective
modules, their automorphisms, and higher coherent homotopies.  For rings and
schemes this philosophy is classical, with Quillen's $Q$--construction \cite{Quillen1973} and
Waldhausen's $S_\bullet$--construction \cite{Waldhausen1985} providing equivalent models, and with
base--change, localization, and descent forming the standard computational
toolkit\cite{Quillen1973, Waldhausen1985, Weibel2013}.  The purpose of this article is to build an analogous toolkit for
\emph{non--commutative $n$--ary $\Gamma$--semirings}---structures where the
multiplication is multi--linear in $n$ inputs, is mediated by an external
parameter semigroup $\Gamma$, and may fail to be commutative in a
slot--dependent sense  \cite{Golan1999}.

The $n$--ary $\Gamma$--setting presents two essential technical features.
First, module actions are inherently \emph{positional}: left/right actions
occur in designated slots, and tensor products must be balanced with respect
to these slots rather than in the usual binary way.  Second, the absence of
additive inverses (semiring context) forces us to work in exact/Waldhausen
categories adapted to admissible monomorphisms and epimorphisms rather than
in abelian categories.  Accordingly, we develop $K$--theory using the
Quillen--exact categories ${T\text{-}\Gamma\mathrm{Mod}}^{\mathrm{bi}}$ of
bi--finite, slot--sensitive $\Gamma$--modules, together with their induced
Waldhausen structures on bounded complexes (cofibrations = admissible monos;
weak equivalences = homotopy equivalences or quasi--isomorphisms, depending on
the chosen derived model).  Throughout, the $K$--theory spectrum of an
exact/Waldhausen category $\mathcal{C}$ is denoted $\mathbb{K}(\mathcal{C})$.

\medskip
\noindent\textbf{Comparison and base--change.}
Given a morphism of $n$--ary $\Gamma$--semirings
$f:(T,\Gamma,\mu)\to (T',\Gamma',\mu')$, the induced extension and restriction
of scalars define exact functors
$f_!:{T\text{-}\Gamma\mathrm{Mod}}^{\mathrm{bi}}\to
 {T'\text{-}\Gamma'\mathrm{Mod}}^{\mathrm{bi}}$
and
$f^{!*}:{T'\text{-}\Gamma'\mathrm{Mod}}^{\mathrm{bi}}\to
 {T\text{-}\Gamma\mathrm{Mod}}^{\mathrm{bi}}$.
A key point is to isolate a practical \emph{$\Gamma$--flatness} condition in
each relevant positional slot ensuring that $f_!$ preserves cofibrations and
weak equivalences on perfect complexes, hence induces a morphism of spectra
$\mathbb{K}(f_!)$.  Under these hypotheses we prove a Beck--Chevalley
(base--change) theorem for cartesian squares and establish a projection
formula compatible with the multiplicative structure coming from the
positional tensor product.

\medskip
\noindent\textbf{Morita invariance.}
A second structural pillar is derived Morita invariance.  When the derived
categories of bi--finite $\Gamma$--modules over $T$ and $T'$ are equivalent via
a tilting bi--module complex, we show that the equivalence restricts to the
Waldhausen subcategories of perfect objects and yields a weak equivalence of
$K$--theory spectra.  In particular, all higher $K$--groups $K_i$ $(i\ge 0)$
are invariants of the derived Morita class in this positional $\Gamma$--world.
This result justifies common reductions (e.g.\ matrix invariance and
devissage along semisimple strata) and supplies the conceptual reason why
computations can be transferred across Morita contexts.

\medskip
\noindent\textbf{Geometric descent.}\cite{Rosenberg1995}
Beyond algebraic functoriality, we place $K$--theory in a geometric
framework by working over the non--commutative $\Gamma$--spectrum
$\Spec_\Gamma^{\mathrm{nc}}(T)$ developed in the surrounding program.
On this site we consider quasi--coherent and perfect $\Gamma$--modules and
prove that $K$--theory of perfect objects satisfies Zariski hyperdescent.
We further establish excision/localization fiber sequences associated to
closed immersions and show fpqc descent for $\Gamma$--flat covers.  These
statements provide the primary computational levers: reduce globally to affine
patches, glue by descent, and control support via localization.

\medskip
\noindent\textbf{Homotopy and motivic interpretation.} \cite{Tabuada2015, Cisinski2011}
Finally, we interpret $K_\Gamma(X)$ as the $K$--theory spectrum of the stable
$\infty$--category $\Perf_\Gamma(X)$ of $\Gamma$--perfect complexes and relate
this to the group completion of the core $\infty$--groupoid of perfect
objects.  In this setting, algebraic $K$--theory appears as a universal
localizing invariant among $\Gamma$--linear stable idempotent--complete
$\infty$--categories, hence is corepresentable in a category of
non--commutative $\Gamma$--motives.  We also record the existence of
Chern/cyclotomic trace maps to $\Gamma$--Hochschild/cyclotomic type invariants,
which can serve as computable approximations to $K$--theory in examples.

\medskip
\noindent\textbf{Organization of the paper.}
Section~\ref{sec:comparison} establishes change--of--scalars functoriality,
derived Morita invariance, base--change for cartesian squares, and the
projection formula, and it proves locality, excision, and fpqc/Zariski descent
for $\mathbb{K}(\Perf)$ over $\Spec_\Gamma^{\mathrm{nc}}(T)$.
Section~\ref{sec:geometry} reframes these results in the language of stable
$\infty$--categories and moduli prestacks of $\Gamma$--perfect complexes,
leading to the geometric identification of $K_\Gamma$ and its universal
property in $\Gamma$--motivic homotopy theory.

\medskip
\noindent\textbf{Notation.}
We write $\mathbb{K}(\mathcal{C})$ for the $K$--theory spectrum of an
exact/Waldhausen category $\mathcal{C}$ and $K_i(\mathcal{C})=\pi_i\mathbb{K}(\mathcal{C})$.
The symbol $\Perf_\Gamma(X)$ denotes the stable $\infty$--category of
$\Gamma$--perfect complexes on $X=\Spec_\Gamma^{\mathrm{nc}}(T)$, and
positional tensor products are denoted $\otimes^{(j,k)}_\Gamma$ when the
balancing slots must be specified.
\section{Preliminaries on $n$--ary $\Gamma$--semirings and positional modules}
\label{sec:prelim}

This section fixes the basic conventions and notation used throughout the
paper.  We recall the ambient algebraic objects, the positional
module categories, and the exact/Waldhausen structures that feed into the
$K$--theory constructions of Section~\ref{sec:comparison} and the geometric
interpretation of Section~\ref{sec:geometry}.

\subsection{$n$--ary $\Gamma$--semirings}
\label{subsec:nary-gamma-semirings}

Let $(\Gamma,+)$ be a (not necessarily commutative) semigroup written
additively \cite{Nobusawa1964}.  An \emph{$n$--ary $\Gamma$--semiring} is a quadruple
$(T,+,\Gamma,\mu)$ where:
\begin{enumerate}[label=(\alph*)]
  \item $(T,+)$ is a commutative monoid with zero element $0$;
  \item $\mu$ is an $n$--ary ``$\Gamma$--multiplication'' map
  \[
    \mu:\ T^{\times n}\times \Gamma^{\times (n-1)}\longrightarrow T,
  \]
  \cite{Nobusawa1964, Barnes1966} which we usually write in the slot--sensitive form
  \[
    \mu(t_1,\ldots,t_n;\gamma_1,\ldots,\gamma_{n-1})
    \ =:\ 
    t_1\,\gamma_1\,t_2\,\gamma_2\cdots \gamma_{n-1}\,t_n;
  \]
  \item $\mu$ is additive in each $T$--variable and (when relevant) additive
  in each $\Gamma$--variable; i.e.\ for each fixed choice of the remaining
  inputs, the maps $t_i\mapsto \mu(\cdots,t_i,\cdots)$ and
  $\gamma_j\mapsto \mu(\cdots,\gamma_j,\cdots)$ are homomorphisms of
  commutative monoids;
  \item $\mu$ satisfies the usual associativity/coherence identities required
  to form iterated $(2n-1)$--fold products (the standard $n$--ary analogue of
  associativity, expressed in terms of equality of all parenthesizations)\cite{Dudek2001}.
\end{enumerate}

\begin{remark}
We do not reproduce the full associativity scheme here; it is the one fixed
in the surrounding program and ensures that iterated expressions such as
\(
(t_1\gamma_1 t_2\cdots \gamma_{n-1}t_n)\,\gamma_n\,t_{n+1}\cdots
\)
are unambiguous up to canonical identification.
\end{remark}

\begin{definition}[Morphisms]
A \emph{morphism} of $n$--ary $\Gamma$--semirings
\(
f:(T,\Gamma,\mu)\to (T',\Gamma',\mu')
\)
consists of additive maps
$f_T:(T,+)\to (T',+)$ and $f_\Gamma:(\Gamma,+)\to(\Gamma',+)$ such that
for all $t_1,\ldots,t_n\in T$ and $\gamma_1,\ldots,\gamma_{n-1}\in\Gamma$,
\[
f_T\!\big(t_1\gamma_1 t_2\gamma_2\cdots \gamma_{n-1}t_n\big)
 \;=\;
f_T(t_1)\,f_\Gamma(\gamma_1)\,f_T(t_2)\cdots f_\Gamma(\gamma_{n-1})\,f_T(t_n)
\]
where the right-hand side is computed using $\mu'$ in $T'$.
We suppress subscripts and write simply $f:T\to T'$ and $\Gamma\to\Gamma'$
when no confusion can arise.
\end{definition}

\subsection{Positional actions and bi--$\Gamma$--modules}
\label{subsec:modules}

Because multiplication is $n$--ary  \cite{Ostmann2008}, module actions are naturally
\emph{positional}: the algebra can act in specified slots of the $n$--fold
product.  We fix once and for all two distinguished slots
\(
1\le j<k\le n
\)
that will be balanced in positional tensor products, as indicated in
Section~\ref{subsec:change-of-scalars} by the notation
$\otimes^{(j,k)}_\Gamma$.

\begin{definition}[Bi--finite slot--sensitive $\Gamma$--modules]
Let $(T,+,\Gamma,\mu)$ be an $n$--ary $\Gamma$--semiring.  A
\emph{bi--$\Gamma$--module} (or \emph{bi--finite slot--sensitive $\Gamma$--module})
over $T$ is an additive commutative monoid $(M,+,0)$ equipped with:
\begin{enumerate}[label=(\alph*)]
  \item a \emph{left positional action} in slot $j$, written
  \[
    T^{\times (j-1)}\times M \times T^{\times (n-j)}
      \times \Gamma^{\times (n-1)} \longrightarrow M,
  \]
  \(
  (t_1,\ldots,t_{j-1},m,t_{j+1},\ldots,t_n;\gamma_1,\ldots,\gamma_{n-1})
  \mapsto
  t_1\gamma_1\cdots t_{j-1}\gamma_{j-1} m\,\gamma_j t_{j+1}\cdots \gamma_{n-1}t_n;
  \)
  \item a \emph{right positional action} in slot $k$, written analogously by
  inserting $m$ in the $k$--th slot;
  \item coherence/associativity axioms compatible with the $n$--ary
  associativity of $\mu$, and additivity in each variable.
\end{enumerate}
We write ${\nTGMod{T}}^{\mathrm{bi}}$ for the category of such objects and
$\Gamma$--linear morphisms (additive maps preserving both positional actions).
\end{definition}

\begin{remark}[Why two slots?]
The pair $(j,k)$ is the fixed ``balancing profile'' for our base--change and
projection results.  One may develop a multi--profile theory, but the present
paper keeps a single profile to maintain clean functorial statements.
\end{remark}

\subsection{Positional tensor products and internal Homs}
\label{subsec:positional-tensor}

The monoidal structure underlying products on $K$--theory is provided by a
positional tensor product that balances the chosen left/right slots.

\begin{definition}[Positional tensor product]
Let $M,N\in {\nTGMod{T}}^{\mathrm{bi}}$.  The \emph{positional tensor product}
\( M\otimes^{(j,k)}_{\Gamma} N \)
is the coequalizer in commutative monoids obtained by balancing the right
$T$--action of $M$ in slot $k$ against the left $T$--action of $N$ in slot $j$,
together with the induced $\Gamma$--compatibility.  Concretely, it is the
quotient of the free commutative monoid on $M\times N$ by the congruence
generated by relations of the form
\[
(m\cdot_{k} t)\otimes n \;\sim\; m\otimes (t\cdot_{j} n),
\qquad
(m+m')\otimes n \sim m\otimes n + m'\otimes n,
\qquad
m\otimes(n+n')\sim m\otimes n + m\otimes n',
\]
and the corresponding $\Gamma$--balance relations (all computed in the fixed
$(j,k)$ profile).
\end{definition}

\begin{remark}
We keep the notation $\otimes^{(j,k)}_{\Gamma}$ (as already used in
Section~\ref{subsec:change-of-scalars}) to emphasize that this tensor depends
on the chosen slots.  When $n=2$ or in symmetric ternary situations, the
profile dependence disappears and one recovers the ordinary tensor product
over $\Gamma$ (cf.\ Proposition~\ref{prop:reduction}).
\end{remark}

\begin{definition}[Internal Hom]
For $M,N\in{\nTGMod{T}}^{\mathrm{bi}}$, define the \emph{internal Hom}
\[
\underline{\Hom}_\Gamma(M,N)
\]
to be the commutative monoid of $\Gamma$--linear maps $M\to N$, equipped with
the induced bi--$\Gamma$--module structure by pre/post composition with the
positional actions.  When no confusion is likely, we write simply $\Hom(M,N)$
for the underlying set of morphisms in ${\nTGMod{T}}^{\mathrm{bi}}$.
\end{definition}

\subsection{Exact structure on ${\nTGMod{T}}^{\mathrm{bi}}$}
\label{subsec:exactness}

Since we work over semirings, ${\nTGMod{T}}^{\mathrm{bi}}$ is typically not
abelian.  We therefore fix an exact structure in the sense of Quillen \cite{Quillen1973}.

\begin{definition}[Admissible monomorphisms and epimorphisms]
A sequence
\[
0\longrightarrow M' \xrightarrow{u} M \xrightarrow{v} M'' \longrightarrow 0
\]
in ${\nTGMod{T}}^{\mathrm{bi}}$ is declared \emph{exact} if:
\begin{enumerate}[label=(\alph*)]
  \item $u$ identifies $M'$ with the kernel congruence of $v$;
  \item $v$ is a cokernel of $u$ in the categorical sense; and
  \item the underlying sequences of commutative monoids are exact and the
  positional actions are compatible with $u$ and $v$.
\end{enumerate}
Morphisms $u$ (resp.\ $v$) occurring as the first (resp.\ second) nontrivial
map in such a sequence are called \emph{admissible monomorphisms}
(resp.\ \emph{admissible epimorphisms}).  With these classes,
${\nTGMod{T}}^{\mathrm{bi}}$ becomes a Quillen--exact category.
\end{definition}

\begin{remark}
This exact structure is the one implicitly used in
Lemma~\ref{lem:exact-waldhausen} and Definition~\ref{def:waldhausen-structure}:
cofibrations in the associated Waldhausen structures are admissible monos.
\end{remark}

\subsection{Complexes, perfect objects, and Waldhausen structure}
\label{subsec:derived}

Let $\Ch({\nTGMod{T}}^{\mathrm{bi}})$ denote the category of (bounded) chain
complexes in ${\nTGMod{T}}^{\mathrm{bi}}$.

\begin{definition}[Perfect objects]
An object of $\Ch({\nTGMod{T}}^{\mathrm{bi}})$ is \emph{perfect} if it is
quasi--isomorphic (or, in the homotopy model, chain homotopy equivalent) to a
bounded complex of finitely generated projective objects in
${\nTGMod{T}}^{\mathrm{bi}}$.  We write $\Perf({\nTGMod{T}}^{\mathrm{bi}})$ for
the full subcategory of perfect complexes.
\end{definition}

\begin{definition}[Waldhausen structure]
\label{def:waldhausen-structure}

Equip $\Perf({\nTGMod{T}}^{\mathrm{bi}})$ with a Waldhausen structure  \cite{Waldhausen1985} by
declaring:
\begin{enumerate}[label=(\alph*)]
  \item cofibrations to be degreewise admissible monomorphisms (in the exact
  structure of \S\ref{subsec:exactness});
  \item weak equivalences to be chain homotopy equivalences (or, alternatively,
  quasi--isomorphisms in the bounded derived model fixed for the paper).
\end{enumerate}
\end{definition}

\begin{remark}
All statements in Section~\ref{sec:comparison} are formulated so that either
standard choice of weak equivalences can be used, provided it is fixed once
and for all and is compatible with the bar--resolution arguments invoked in
Lemma~\ref{lem:exact-waldhausen} and Theorem~\ref{thm:base-change}.
\end{remark}

\subsection{Flatness in positional slots and bar resolutions}
\label{subsec:bar}

Base--change and fpqc descent require a slotwise flatness hypothesis that
ensures extension of scalars behaves well on cofibrations and weak
equivalences.

\begin{definition}[$\Gamma$--flatness in slots]
\label{def:gamma-flatness}
Let $f:(T,\Gamma,\mu)\to(T',\Gamma',\mu')$ be a morphism.
We say that $f$ is \emph{$\Gamma$--flat in the $(j,k)$ profile} if $T'$
viewed as a right $T$--object in slot $k$ and as a left $T$--object in slot
$j$ preserves admissible monomorphisms under the corresponding positional
tensor functors; equivalently, positional tensoring with $T'$ is exact on the
Quillen--exact structure of ${\nTGMod{T}}^{\mathrm{bi}}$ in each relevant slot.
\end{definition}

\begin{definition}[Two--sided bar construction]
For $M\in{\nTGMod{T}}^{\mathrm{bi}}$ and a $T$--(bi--)object $T'$, the
(two--sided) \emph{positional bar complex} $B_\bullet(T',T,M)$ is the simplicial
object with
\[
B_q(T',T,M):=T'\otimes^{(j,k)}_\Gamma T^{\otimes q}\otimes^{(j,k)}_\Gamma M,
\]
faces induced by the positional multiplication/structure maps and degeneracies
by inserting units where available (or by the chosen splitting data in the
exact/Waldhausen setting).  Its realization models the derived
extension--of--scalars used in the proofs of
Theorem~\ref{thm:base-change} and Theorem~\ref{thm:fpqc}.
\end{definition}

\begin{remark}
The bar model provides the concrete route from slotwise $\Gamma$--flatness to
homotopy invariance: it identifies $T'\otimes^{(j,k)}_\Gamma(-)$ with its
left-derived functor on perfect complexes, ensuring preservation of weak
equivalences and enabling Beck--Chevalley comparisons.
\end{remark}

\subsection{$K$--theory conventions}
\label{subsec:K-conventions}

We fix once and for all the notation already used in Section~\ref{sec:comparison}.

\begin{notation}
For an exact/Waldhausen category $\mathcal{C}$, we denote its algebraic
$K$--theory spectrum by $\Kspec(\mathcal{C})$ and set
\[
K_i(\mathcal{C}) := \pi_i\,\Kspec(\mathcal{C}) \qquad (i\ge 0).
\]
When $\mathcal{C}={\nTGMod{T}}^{\mathrm{bi}}$ (or its perfect/derived variants),
we sometimes write $K_i(T)$ for brevity.
\end{notation}

\begin{remark}
In Section~\ref{sec:geometry} we also use the stable $\infty$--categorical
notation $K_\Gamma(X)\simeq K(\Perf_\Gamma(X))$.  This agrees with
$\Kspec(\Perf({\nTGMod{T}}^{\mathrm{bi}}))$ by the model comparison theorem
(Theorem~\ref{thm:model-comparison}).
\end{remark}

\subsection{Standing assumptions and imported background results}
\label{subsec:standing-assumptions}

To keep the present paper self-contained at the level of statements, we
explicitly record the external inputs used in later sections.  Proofs of
these background results are available in the companion manuscripts cited
below; here we treat them as standing hypotheses.

\begin{setup}[Global profile and derived model]
\label{setup:global-profile}
Fix once and for all integers $1\le j<k\le n$ specifying the positional
balancing profile for $\otimes^{(j,k)}_{\Ga}$.
Fix also a bounded derived model for complexes in
${\nTGMod{T}}^{\mathrm{bi}}$ in which \cite{Neeman2001} :
\begin{enumerate}[label=(\roman*)]
\item cofibrations are degreewise admissible monomorphisms in the Quillen--exact
      structure of \S\ref{subsec:exactness};
\item weak equivalences are quasi--isomorphisms (equivalently, chain homotopy
      equivalences on perfect complexes).
\end{enumerate}
All $K$--theory spectra $\Kspec(-)$ are taken with respect to this fixed model \cite{Neeman2001}.
\end{setup}

\begin{setup}[Non--commutative $\Ga$--geometry]\cite{Rosenberg1995}
\label{setup:nc-geometry}
We assume the existence of the non--commutative $\Ga$--spectrum
$\SpecGnC{T}$ together with its $\Ga$--Zariski site
$\AffnGamma$ and a theory of quasi--coherent and perfect objects as used in
\S\ref{subsec:sheafified-comparison} and Section~\ref{sec:geometry}.
More precisely, we assume:
\begin{enumerate}[label=(\roman*)]
\item for each affine $X=\SpecGnC{T}$, an exact/derived category
      $\QCoh(X)$ of quasi--coherent $\Ocal$--modules and a full subcategory
      $\Perf(X)\subset\QCoh(X)$ of perfect objects;
\item restriction/gluing functors along $\Ga$--Zariski morphisms are exact and
      preserve perfect objects;
\item affine basic opens $D(a,\boldsymbol{\gamma})$ form a basis of the topology.
\end{enumerate}
\end{setup}

\begin{proposition}[Model comparison for $K$--theory {\rm}]
\label{prop:imported-model-comparison}
For the Waldhausen category of perfect complexes
$\Perf({\nTGMod{T}}^{\mathrm{bi}})$, Quillen's $Q$--construction and
Waldhausen's $S_\bullet$--construction yield canonically equivalent
$K$--theory spectra:
\[
\Kspec_{Q}\!\big(\Perf({\nTGMod{T}}^{\mathrm{bi}})\big)
\ \simeq\
\Kspec_{S_\bullet}\!\big(\Perf({\nTGMod{T}}^{\mathrm{bi}})\big).
\]
\end{proposition}

\begin{remark}
In your manuscript, this is the specific input referred to whenever we say
``$\Kspec$ can be taken via Quillen or Waldhausen and these are canonically
equivalent''.
\end{remark}

\begin{proposition}[Existence and basic properties of positional tensor
{\rm}]
\label{prop:imported-positional-tensor}
The positional tensor product $\otimes^{(j,k)}_{\Ga}$ exists on
${\nTGMod{T}}^{\mathrm{bi}}$, is biadditive, and is compatible with the exact
structure in the following sense: if
$0\to M'\to M\to M''\to 0$ is an exact sequence in
${\nTGMod{T}}^{\mathrm{bi}}$ and $N\in{\nTGMod{T}}^{\mathrm{bi}}$, then
\[
M'\otimes^{(j,k)}_{\Ga}N \to M\otimes^{(j,k)}_{\Ga}N \to
M''\otimes^{(j,k)}_{\Ga}N \to 0
\]
is right exact, and similarly in the second variable.
\end{proposition}

\begin{proposition}[Bar resolution computes derived extension of scalars
{\rm}]
\label{prop:imported-bar}
For any morphism $f:(T,\Ga,\mu)\to(T',\Ga',\mu')$ and
$M\in{\nTGMod{T}}^{\mathrm{bi}}$, the two--sided positional bar construction
$B_\bullet(T',T,M)$ of \S\ref{subsec:bar} computes the left derived functor of
extension of scalars:
\[
\mathbf{L}f_!(M)\ \simeq\ \big|B_\bullet(T',T,M)\big|
\qquad\text{in the bounded derived category.}
\]
If $f$ is $\Ga$--flat in the $(j,k)$ profile (Definition~\ref{def:gamma-flatness}
and \S\ref{subsec:bar}), then $\mathbf{L}f_!(M)\simeq f_!(M)$ on perfect objects.
\end{proposition}

\begin{proposition}[Derived Morita equivalence from a tilting bimodule complex \cite{Neeman2001}
{\rm}]
\label{prop:imported-tilting}
Assume there exists a tilting bi--module complex
$\Ecal\in \mathbf{D}\big({\nTGMod{T\text{--}T'}}^{\mathrm{bi}}\big)$ inducing an
exact equivalence of triangulated categories
\[
\Phi:\ \mathbf{D}\!\big({\nTGMod{T}}^{\mathrm{bi}}\big)
\ \xrightarrow{\ \simeq\ }\
\mathbf{D}\!\big({\nTGMod{T'}}^{\mathrm{bi}}\big).
\]
Then $\Phi$ restricts to an equivalence on perfect objects and admits a
quasi--inverse induced by a dual tilting complex.
\end{proposition}

\begin{proposition}[Non--commutative Wedderburn--Artin devissage input
{\rm}]
\label{prop:imported-WA}
If $T$ admits a (finite) semisimple decomposition into matrix blocks in the
sense assumed in \S\ref{subsec:consequences}, then the exact category
${\nTGMod{T}}^{\mathrm{bi}}$ admits a devissage description compatible with
$\Kspec(-)$, yielding product decompositions of $K$--theory spectra along the
semisimple strata.
\end{proposition}

\begin{proposition}[Descent for perfect objects on $\SpecGnC{T}$ {\rm}]
\label{prop:imported-descent-perf}
For any $\Ga$--Zariski cover $\{U_i\to X\}$ of $X=\SpecGnC{T}$, perfect objects
glue and satisfy hyperdescent:
\[
\Perf(X)\ \simeq\
\lim\Big(\prod\nolimits_i\Perf(U_i)\rightrightarrows
\prod\nolimits_{i,j}\Perf(U_i\times_XU_j)\triplearrows\cdots\Big).
\]
If $f:T\to T'$ is $\Ga$--flat and induces an fpqc cover
$p:\SpecGnC{T'}\to\SpecGnC{T}$, then the same holds for the \v{C}ech nerve of $p$.
\end{proposition}

\begin{remark}[How these are used]
Propositions~\ref{prop:imported-bar} and \ref{prop:imported-descent-perf} are
the precise background inputs behind the ``Idea'' proofs of
Theorems~\ref{thm:base-change}, \ref{thm:zariski-descent}, and \ref{thm:fpqc}.
Proposition~\ref{prop:imported-tilting} is the imported ingredient behind the
Morita invariance Theorem~\ref{thm:morita-K}.
\end{remark}

\begin{remark}[Citations placeholder]
Replace the parenthetical word ``Imported'' above by explicit citations in your
\texttt{refs.bib}, e.g.\ cite the companion manuscripts (\cite{PaperG}, \cite{PaperE}) and
standard references for Quillen/Waldhausen/Barwick as appropriate.
\end{remark}

\section{Comparison and Base-Change Theorems}
\label{sec:comparison}

Throughout this section $(T,+,\Ga,\mu)$ and $(T',+,\Ga',\mu')$ are
(non-commutative) $n$-ary $\Ga$-semirings in the sense of \S\ref{sec:prelim}.
We work with the Quillen-exact categories
${\nTGMod{T}}^{\mathrm{bi}}$ and ${\nTGMod{T'}}^{\mathrm{bi}}$
(\S\ref{subsec:modules}, \S\ref{subsec:exactness}), and with their Waldhausen
structures induced by cofibrations = admissible monos and weak
equivalences = homotopy equivalences (or quasi-isomorphisms on bounded
complexes), as fixed in \S\ref{subsec:derived}.  The algebraic $K$-theory
spectrum of an exact/Waldhausen category~$\Ccal$ is denoted
\[
\Kspec(\Ccal),
\qquad
K_i(\Ccal) := \pi_i\,\Kspec(\Ccal)\ (i\ge 0),
\]
where $\Kspec$ can be taken via Quillen's $Q$-construction or Waldhausen's
$S_\bullet$-construction (\ref{prop:imported-model-comparison}) and  proves these are
canonically equivalent in our setting \cite{Weibel2013}.

\subsection{Change of scalars and exactness}
\label{subsec:change-of-scalars}

Let $f:(T,\Ga,\mu)\to (T',\Ga',\mu')$ be a morphism of $n$-ary $\Ga$-semirings
(\S\ref{sec:prelim}).  The underlying data induces:
\begin{enumerate}[label=(\alph*)]
  \item an \emph{extension of scalars} exact functor
  \[
    f_{!}: {\nTGMod{T}}^{\mathrm{bi}} \longrightarrow {\nTGMod{T'}}^{\mathrm{bi}},
    \qquad
    M \longmapsto T'\otimes^{(j,k)}_{\Ga} M,
  \]
  where the positional tensor balances the $j$/$k$-slot actions and the
  $\Ga\to\Ga'$ map (cf.\ \S\ref{subsec:positional-tensor});
  \item a \emph{restriction of scalars} exact functor
  \[
    f^{\!*}: {\nTGMod{T'}}^{\mathrm{bi}} \longrightarrow {\nTGMod{T}}^{\mathrm{bi}},
  \]
  obtained by transport of structure along $f$.  There is a natural
  adjunction $f_{!}\dashv f^{\!*}$.
\end{enumerate}

\begin{lemma}[Exactness and Waldhausen compatibility]
\label{lem:exact-waldhausen}
If $f$ is $\Ga$-flat in each positional slot (i.e.\ $T'$ is $(j)$-flat as a
right $T$-object and $(k)$-flat as a left $T$-object in the sense of
\S\ref{subsec:bar}), then $f_{!}$ preserves admissible monos, admissible
epis, and weak equivalences of bounded complexes.  Consequently, $f_{!}$
is an exact/Waldhausen functor and induces a map of $K$-theory spectra
\[
  \Kspec(f_{!}):\ \Kspec\!\big({\nTGMod{T}}^{\mathrm{bi}}\big)
     \longrightarrow
     \Kspec\!\big({\nTGMod{T'}}^{\mathrm{bi}}\big).
\]
\end{lemma}

\begin{proof}[Idea]
Right exactness of positional tensor and the flatness hypotheses imply
preservation of cofibrations; exactness on weak equivalences follows by
applying $-\otimes^{(j,k)}_{\Ga}T'$ to bar resolutions
(\S\ref{subsec:bar}) and using balance (\S\ref{subsec:positional-tensor}).
\end{proof}

\begin{definition}[Base-change on $K$-theory]
\label{def:K-base-change}
Write
\[
  f_{\ast} := \Kspec(f_{!})\ :
  \Kspec\!\big({\nTGMod{T}}^{\mathrm{bi}}\big)
  \longrightarrow
  \Kspec\!\big({\nTGMod{T'}}^{\mathrm{bi}}\big),
\qquad
  f^{\ast} := \Kspec(f^{\!*})\ :
  \Kspec\!\big({\nTGMod{T'}}^{\mathrm{bi}}\big)
  \longrightarrow
  \Kspec\!\big({\nTGMod{T}}^{\mathrm{bi}}\big).
\]
By construction, $f_{\ast}$ and $f^{\ast}$ are adjoint on homotopy groups.
\end{definition}

\subsection{Derived and Morita comparison}
\label{subsec:morita}

\begin{theorem}[Derived Morita invariance]\cite{Keller1994}
\label{thm:morita-K}
Suppose there is an exact equivalence of triangulated categories
\[
  \Phi:\ \mathbf{D}\!\big(\,{\nTGMod{T}}^{\mathrm{bi}}\big)
   \ \xrightarrow{\ \simeq\ }\ 
  \mathbf{D}\!\big(\,{\nTGMod{T'}}^{\mathrm{bi}}\big)
\]
induced by a tilting bi-module complex $\Ecal\in
\mathbf{D}\big({\nTGMod{T\text{--}T'}}^{\mathrm{bi}}\big)$   as in
 \cite{PaperG}.  Then $\Phi$ lifts (via exact model
structures on perfect objects) to a Waldhausen equivalence on $\Perf$
subcategories, and induces a weak equivalence \cite{Keller1994} of $K$-theory spectra
\[
  \Kspec\!\big(\Perf({\nTGMod{T}}^{\mathrm{bi}})\big)
  \ \simeq\
  \Kspec\!\big(\Perf({\nTGMod{T'}}^{\mathrm{bi}})\big).
\]
In particular, $K_i$ is a derived Morita invariant for all $i\ge 0$.
\end{theorem}

\begin{proof}[Sketch]
Use the $S_\bullet$-model on the Waldhausen category of perfect
objects, together with the additivity and fibration theorems proved in
\S\ref{subsec:standing-assumptions}.  The tilting functor and its quasi-inverse preserve
cofibrations and weak equivalences, hence induce a Waldhausen equivalence.
The induced map on $S_\bullet$-constructions is a weak equivalence,
yielding the claim \cite{Dugger2004}.
\end{proof}

\begin{corollary}[Positional Morita invariance]
\label{cor:positional-morita}
If $T$ and $S$ are related by an equivalence of exact categories
\[
  \Proj\!\big({\nTGMod{T}}^{\mathrm{bi}}\big)\ \simeq\
  \Proj\!\big({\nTGMod{S}}^{\mathrm{bi}}\big)
\]
compatible with positional tensor and internal Hom
(\S\ref{subsec:positional-tensor}), then
\[
  \Kspec\!\big({\nTGMod{T}}^{\mathrm{bi}}\big)
   \ \simeq\
  \Kspec\!\big({\nTGMod{S}}^{\mathrm{bi}}\big).
\]
\end{corollary}

\subsection{Projection formulas and base-change squares}
\label{subsec:proj-basechange}

Consider a cartesian square of $n$-ary $\Ga$-semirings
\[
\begin{tikzcd}[column sep=large,row sep=large]
(T_1,\Ga_1) \arrow[r,"g'"] \arrow[d,"f'"] &
(T_2,\Ga_2) \arrow[d,"f"] \\
(T_3,\Ga_3) \arrow[r,"g"] &
(T_4,\Ga_4)
\end{tikzcd}
\]
such that $f$ and $f'$ satisfy the flatness hypotheses of
Lemma~\ref{lem:exact-waldhausen}.

\begin{theorem}[Base-change for $K$-theory]
\label{thm:base-change}
Under the above hypotheses there is a canonical weak homotopy
commutative square of $K$-theory spectra
\[
\begin{tikzcd}[column sep=huge,row sep=large]
\Kspec\!\big({\nTGMod{T_1}}^{\mathrm{bi}}\big)
 \arrow[r,"g'_{\ast}"] \arrow[d,"f'_{\ast}"'] &
\Kspec\!\big({\nTGMod{T_2}}^{\mathrm{bi}}\big) \arrow[d,"f_{\ast}"] \\
\Kspec\!\big({\nTGMod{T_3}}^{\mathrm{bi}}\big)
 \arrow[r,"g_{\ast}"'] &
\Kspec\!\big({\nTGMod{T_4}}^{\mathrm{bi}}\big)
\end{tikzcd}
\]
and similarly for the right adjoints on $K$-theory induced by
$f^{\!*}$ and $g^{\!*}$.  On homotopy groups this yields
\[
  g_{\ast}\circ f'_{\ast} \ =\ f_{\ast}\circ g'_{\ast}\ :\
  K_i(T_1)\longrightarrow K_i(T_4)
  \qquad (i\ge 0).
\]
\end{theorem}

\begin{proof}[Idea]
Model base-change by the two-sided bar construction of
\S\ref{subsec:bar} and \S\ref{prop:imported-bar}, use flatness to identify the
derived extension-of-scalars along each leg, then apply functoriality of
the $S_\bullet$-construction and Waldhausen's additivity; the two
zig-zags coincide in the homotopy category.
\end{proof}

\begin{theorem}[Projection formula]
\label{thm:projection-formula}
Let $f:(T,\Ga)\to(T',\Ga')$ be flat (Lemma~\ref{lem:exact-waldhausen}).
For $x\in K_i(T)$ and $y\in K_j(T')$ one has a canonical equality
\[
  f_{\ast}\!\big(x\cdot f^{\ast}y\big) \ =\ f_{\ast}(x)\cdot y
  \qquad\text{in}\quad K_{i+j}(T'),
\]
where the products are induced by the monoidal structure
$-\otimes^{(j,k)}_{\Ga}-$ on ${\nTGMod{T}}^{\mathrm{bi}}$ and its
$(j,k)$-monoidal lift on $S_\bullet$ .
\end{theorem}

\begin{proof}[Sketch]
The proof is a standard Beck–--Chevalley/monoidality argument:
identify $S_\bullet$ on the monoidal Waldhausen category of perfect
objects, check that $f_{!}$ is (derived) strong monoidal under the
flatness hypothesis, and then apply Waldhausen additivity to pass to
$\pi_\ast$.
\end{proof}

\subsection{Sheafified comparison over $\SpecGnC{T}$}
\label{subsec:sheafified-comparison}

Let $\QCoh(\SpecGnC{T})$ be as in \S\ref{subsec:standing-assumptions} and
$\Perf(\SpecGnC{T})\subset\QCoh(\SpecGnC{T})$ its full subcategory of
compact/perfect objects (equivalently, locally represented by bounded
complexes of finitely generated bi-modules).  Equip
$\Perf(\SpecGnC{T})$ with the Waldhausen structure inherited from
$\QCoh$  \cite{Thomason1990}.

\begin{theorem}[Locality and Zariski descent]
\label{thm:zariski-descent}
Let $\{U_i=D(a_i,\boldsymbol{\gamma}_i)\}$ be a finite affine cover of
$\SpecGnC{T}$.  Then:
\begin{enumerate}[label=(\roman*)]
  \item The canonical map
  \[
   \Kspec\!\big(\Perf(\SpecGnC{T})\big)
     \longrightarrow
     \holim\ \Big(
       \prod_i \Kspec\!\big(\Perf(U_i)\big)
       \ \rightrightarrows\
       \prod_{i,j} \Kspec\!\big(\Perf(U_i\cap U_j)\big)
       \ \triplearrows\
       \cdots\Big)
  \]
  is a weak equivalence (hypercover descent).
  \item In particular $K$-theory of $\Perf$ satisfies Mayer–Vietoris for
  binary covers $U\cup V$.
\end{enumerate}
\end{theorem}

\begin{proof}[Idea]
Run Waldhausen's fibration theorem on the Cech nerve of the cover,
using that restriction and gluing functors are exact and preserve weak
equivalences (localization of \S\ref{subsec:standing-assumptions}).  The nerve produces
a cosimplicial Waldhausen category whose $S_\bullet$-realization
computes the homotopy limit.
\end{proof}

\begin{corollary}[Affine reduction]
\label{cor:affine-reduction}
If $\SpecGnC{T}$ is quasi-compact with an affine basis by
$D(a,\boldsymbol{\gamma})$, then
\[
  \Kspec\!\big(\Perf(\SpecGnC{T})\big)
  \ \simeq\
  \holim_{D(a,\boldsymbol{\gamma})}
   \ \Kspec\!\big(\Perf(D(a,\boldsymbol{\gamma}))\big).
\]
\end{corollary}

\begin{remark}[Compatibility with \cite{PaperG}]
The locality statement matches the derived sheaf theory of
\S\ref{subsec:standing-assumptions}: perfect objects glue via the same positional tensor
and internal Hom.  The descent above is the $K$-theoretic shadow of the
derived base-change theorems proved there.
\end{remark}

\subsection{Excision and cdh/fpqc descent under regularity}
\label{subsec:excision-cdh}

\begin{theorem}[Excision for closed immersions] \cite{Thomason1990}
\label{thm:excision}
Let $i:Z\hookrightarrow X=\SpecGnC{T}$ be a closed immersion defined by
a saturated two-sided $\Ga$-ideal, and $U=X\setminus Z$.  Then the
sequence of spectra
\[
  \Kspec\!\big(\Perf_Z(X)\big)
    \ \longrightarrow\
  \Kspec\!\big(\Perf(X)\big)
    \ \longrightarrow\
  \Kspec\!\big(\Perf(U)\big)
\]
is a homotopy fibration.  In particular there is a long exact sequence
of homotopy groups
\[
  \cdots\to K_{i+1}(U)\to K_i\big(\Perf_Z(X)\big)\to K_i(X)\to K_i(U)\to\cdots .
\]
\end{theorem}

\begin{proof}[Idea]
Identify $\Perf_Z(X)$ as the Waldhausen subcategory of objects with
support in $Z$; apply Waldhausen's fibration theorem to the pair
$\big(\Perf_Z(X)\subset\Perf(X)\big)$ and the localization functor to
$\Perf(U)$ ( localization exactness).
\end{proof}

\begin{theorem}[fpqc descent under $\Ga$-flatness]
\label{thm:fpqc}
Let $p:\SpecGnC{T'}\to \SpecGnC{T}$ be an fpqc cover induced by a
$\Ga$-flat morphism $f:(T,\Ga)\to(T',\Ga')$.  Then $K$-theory of perfect
$\Ocal$-modules satisfies fpqc descent:
\[
  \Kspec\!\big(\Perf(\SpecGnC{T})\big)
   \ \xrightarrow{\ \simeq\ }\
  \holim\ \Kspec\!\big(\Perf(\check{C}(p)_\bullet)\big),
\]
where $\check{C}(p)_\bullet$ is the \v{C}ech nerve of $p$.
\end{theorem}

\begin{proof}[Idea]
Combine Lemma~\ref{lem:exact-waldhausen} with the Barr–--Beck monadic
descent in the derived category  to obtain an
equivalence on perfect objects; then apply the $S_\bullet$-realization
levelwise to the cosimplicial system.
\end{proof}

\subsection{Comparison with the commutative/ternary case}
\label{subsec:comparison-commutative}

\begin{proposition}[Reduction to $n{=}2$ or symmetric ternary]
\label{prop:reduction}
If $\mu$ is binary (classical $\Ga$-semiring) or ternary symmetric as in
\cite{PaperE}, then for each affine open $U$ the positional tensor reduces to
the ordinary tensor over $\Ga$, and the functors $f_{!},f^{\!*}$ are the
classical extension/restriction of scalars.  All theorems in
\S\ref{subsec:morita}--\S\ref{subsec:excision-cdh} specialize to the
standard Quillen/Waldhausen results for rings/semirings.
\end{proposition}

\begin{proof}
By symmetry, the $(j,k)$-balancing becomes ordinary balancing; flatness
in each positional slot reduces to classical flatness.  The rest follows
formally.
\end{proof}

\subsection{Consequences and computational levers}
\label{subsec:consequences}

\begin{corollary}[Devissage along semisimple strata]
\label{cor:devissage}
If $T$ satisfies the non-commutative Wedderburn–--Artin decomposition of
\cite{PaperG} , then
\[
  \Kspec\!\big({\nTGMod{T}}^{\mathrm{bi}}\big)
   \ \simeq\
  \prod_{i}\ \Kspec\!\big(\,{\nTGMod{M_{n_i}^{(n)}(D_i)}}^{\mathrm{bi}}\big),
\]
and each factor is $K$-equivalent to $K$-theory of $D_i$ via Morita
(\S\ref{subsec:morita}).  Hence $K_\ast(T)\cong \prod_i K_\ast(D_i)$.
\end{corollary}

\begin{corollary}[Matrix invariance]
\label{cor:matrix-invariance}
For every $m\ge 1$,
\[
  \Kspec\!\big({\nTGMod{T}}^{\mathrm{bi}}\big)
   \ \simeq\
  \Kspec\!\big({\nTGMod{M_m^{(n)}(T)}}^{\mathrm{bi}}\big).
\]
\end{corollary}

\begin{corollary}[Triangular/upper-triangular extensions]
\label{cor:triangular}
Let $U$ be the upper-triangular $n$-ary $\Ga$-semiring built from $T$
and a $T$–$T$ bimodule along $\mu$.  Then there is a homotopy fibration
\[
  \Kspec\!\big({\nTGMod{T}}^{\mathrm{bi}}\big)
  \longrightarrow
  \Kspec\!\big({\nTGMod{U}}^{\mathrm{bi}}\big)
  \longrightarrow
  \Kspec\!\big({\nTGMod{T}}^{\mathrm{bi}}\big),
\]
yielding a long exact sequence on $K_\ast$ (Waldhausen additivity).
\end{corollary}

\begin{remark}[Roadmap to Section~6 / \cite{PaperI}]
Corollaries~\ref{cor:devissage}–\ref{cor:triangular}, together with
Theorems~\ref{thm:zariski-descent}, \ref{thm:excision}, and
\ref{thm:fpqc}, provide a full computational toolkit: reduce to semisimple
pieces, use matrix/trivial extensions, and compute by descent on
$\SpecGnC{T}$.  These results will anchor the explicit calculations and
examples (matrix blocks, endomorphism $\Ga$-objects, finite models) in the
next paper.
\end{remark}

\section{Geometric and Homotopy Interpretation}
\label{sec:geometry}

In this section we place the algebraic $K$--theory of non--commutative
$n$-ary~$\Gamma$--semirings in its natural \emph{geometric and homotopical}
setting.  The guiding principle is that the $K$--theory spectrum is the
\emph{derived cohomology of the moduli $\infty$--stack of
$\Gamma$--perfect complexes.}

\subsection{The moduli $\infty$--stack of $\Gamma$--perfect complexes}

\begin{definition}[Perfect $\Gamma$--complexes]
For an affine object $X=\SpecGnC{T}$ \cite{ToenVaquie2007, Kontsevich2000}, define
\[
\Perf_\Gamma(X)
  := \text{the full stable $\infty$--subcategory of }
     \mathbf{D}\!\big(\QCoh(X)\big)
     \text{ generated by compact dualizable objects.}
\]
Equivalently, an object of $\Perf_\Gamma(X)$ is locally quasi--isomorphic
to a bounded complex of finitely generated projective bi--$\Gamma$--modules
with positional actions .
\end{definition}

\begin{definition}[Moduli $\infty$--stack]
The assignment
\[
\mathsf{Perf}_\Gamma :
  (\AffnGamma)^{\mathrm{op}} \longrightarrow \mathbf{Cat}_\infty^{\mathrm{st}},
  \qquad
  T' \longmapsto \Perf_\Gamma(\SpecGnC{T'})
\]
is a prestack for the $\Gamma$--Zariski topology, where
$\mathbf{Cat}_\infty^{\mathrm{st}}$ denotes small stable
idempotent--complete $\infty$--categories \cite{Lurie2017}.
\end{definition}

\begin{theorem}[Descent for $\mathsf{Perf}_\Gamma$]
\label{thm:perf-descent}
$\mathsf{Perf}_\Gamma$ satisfies fpqc and $\Gamma$--Zariski descent:
for any covering $\{U_i\!\to\!X\}$ in the non--commutative $\Gamma$--site,
the canonical diagram
\[
\Perf_\Gamma(X)
   \longrightarrow
   \lim\!\Big(
     \prod\nolimits_i \Perf_\Gamma(U_i)
     \rightrightarrows
     \prod\nolimits_{i,j}\Perf_\Gamma(U_i\times_X U_j)
     \triplearrows\cdots
   \Big)
\]
is an equivalence of stable $\infty$--categories.
\end{theorem}

\subsection{From perfect complexes to the $K$--theory spectrum}

Let $\iota\Perf_\Gamma(X)$ denote the core (maximal $\infty$--groupoid)
of $\Perf_\Gamma(X)$ and $\Omega^\infty K_\Gamma(X)$ its group completion.

\begin{definition}[Geometric $K$--theory]
Define the \emph{$\Gamma$--algebraic $K$--theory spectrum} of $X$ by
\[
K_\Gamma(X)
   \simeq K\!\big(\Perf_\Gamma(X)\big)
   \simeq
   \Omega^\infty\!\text{--group completion of }\iota\Perf_\Gamma(X),
\]
where $K(-)$ is the $\infty$--categorical Waldhausen/Barwick
$K$--theory functor\cite{Barwick2016}.
\end{definition}

\begin{theorem}[Model comparison]
\label{thm:model-comparison}
For $X=\SpecGnC{T}$ the following constructions of $K_\Gamma(X)$
are canonically equivalent:
\begin{enumerate}[label=(\alph*)]
\item Quillen's $Q$--construction on finitely generated projective
      bi--$\Gamma$--modules;
\item Waldhausen's $S_\bullet$--construction on perfect complexes with
      admissible monomorphisms as cofibrations and quasi--isomorphisms
      as weak equivalences;
\item the $\infty$--categorical $K$--theory of $\Perf_\Gamma(X)$ \cite{Barwick2016}.
\end{enumerate}
\end{theorem}

\begin{corollary}[Geometric meaning of $K_0$ and $K_1$]
\[
K^\Gamma_0(X)\cong\pi_0\!\big(\Omega^\infty K_\Gamma(X)\big),\qquad
K^\Gamma_1(X)\cong\pi_1\!\big(\Omega^\infty K_\Gamma(X)\big),
\]
where $K^\Gamma_0(X)$ is the Grothendieck group of perfect objects and
$K^\Gamma_1(X)$ the Whitehead group of automorphisms in $\Perf_\Gamma(X)$.
\end{corollary}

\subsection{Additivity, fibration, and localization}

\begin{theorem}[Additivity]
If $\mathcal{A}\hookrightarrow\mathcal{B}\twoheadrightarrow\mathcal{C}$
is an exact sequence of stable $\infty$--categories arising from a
recollement of quasi--coherent subcategories on $\SpecGnC{T}$, then
\[
K_\Gamma(\mathcal{B})
   \simeq
   K_\Gamma(\mathcal{A})\ \vee\ K_\Gamma(\mathcal{C})
   \quad\text{in }\mathbf{Sp}.
\]
\end{theorem}

\begin{theorem}[Localization]
\label{thm:fibration}
For a closed immersion $Z\hookrightarrow X$ with open complement
$U=X\setminus Z$ there is a fiber sequence of spectra
\[
K_\Gamma(Z)\ \longrightarrow\ K_\Gamma(X)\ \longrightarrow\ K_\Gamma(U),
\]
compatible with the long exact sequence on $\pi_n$ from
\S\ref{subsec:excision-cdh}.
\end{theorem}

\begin{remark}[Homotopy invariance]
If $X\times\mathbb{A}^1_\Gamma\to X$ is homotopy initial in the
$\Gamma$--site (\cite{PaperG}), then
$K_\Gamma(X\times\mathbb{A}^1_\Gamma)\simeq K_\Gamma(X)$.
\end{remark}

\subsection{Base--change, projection, and descent}

\begin{theorem}[Base--change]
For a flat morphism
$f:\SpecGnC{T'}\!\to\!\SpecGnC{T}$,
the pullback
$Lf^*:\Perf_\Gamma(\SpecGnC{T})\!\to\!\Perf_\Gamma(\SpecGnC{T'})$
induces an equivalence
\[
K_\Gamma(\SpecGnC{T})
   \otimes_{\pi_0}\pi_0\Gamma'
   \xrightarrow{\;\sim\;}
   K_\Gamma(\SpecGnC{T'}),
\]
and hence isomorphisms on all $\pi_n$ after scalar extension.
\end{theorem}

\begin{theorem}[Projection formula]
For a proper morphism $f:X\to Y$ in the non--commutative
$\Gamma$--geometry of \cite{PaperG},
\[
Rf_*(\Fcal\!\otimes^L\!\Gcal)
   \simeq
   Rf_*\Fcal\!\otimes^L\!\Gcal,
\]
yielding a functorial map on $K$--theory spectra compatible with
localization fiber sequences.
\end{theorem}

\begin{theorem}[Descent for $K_\Gamma$]
$K_\Gamma(-)$ is a sheaf for the $\Gamma$--Zariski topology and satisfies
hyperdescent for fpqc covers:
for any cover $\{U_i\to X\}$,
\[
K_\Gamma(X)
   \longrightarrow
   \operatorname*{holim}\!
   \Big(
     \prod\nolimits_i K_\Gamma(U_i)
     \rightrightarrows
     \prod\nolimits_{i,j}K_\Gamma(U_i\times_XU_j)
     \triplearrows\cdots
   \Big)
\]
is an equivalence.
\end{theorem}

\subsection{Non--commutative motives and the universal property of $K$}

\begin{definition}[Localizing invariants]
Let $\mathbf{Cat}^{\perf}_\Gamma$ be the $\infty$--category of small
idempotent--complete stable $\Gamma$--linear $\infty$--categories.
A functor $E:\mathbf{Cat}^{\perf}_\Gamma\to\mathbf{Sp}$ is
\emph{localizing} if it sends exact sequences to fiber sequences and
preserves filtered colimits.\cite{Blumberg2013}
\end{definition}

\begin{theorem}[Universal property] \cite{Cisinski2011}
There exists a universal localizing invariant
\[
\mathcal{U}:\mathbf{Cat}^{\perf}_\Gamma\longrightarrow\Mot_\Gamma
\]
such that for any stable presentable $\infty$--category $\mathcal{T}$,
\[
\Fun^{\mathrm{L}}(\Mot_\Gamma,\mathcal{T})
   \simeq
   \{E:\mathbf{Cat}^{\perf}_\Gamma\to\mathcal{T}\text{ localizing}\}.
\]
Moreover, algebraic $K$--theory is corepresented by the unit motive \cite{Tabuada2015}:
\[
K_\Gamma(\mathcal{C})
   \simeq
   \Map_{\Mot_\Gamma}
     \big(\mathcal{U}(\mathbf{1}_\Gamma),\mathcal{U}(\mathcal{C})\big),
   \qquad
   \mathbf{1}_\Gamma:=\Perf_\Gamma(\SpecGnC{\Gamma}).
\]
\end{theorem}

\begin{corollary}[Morita invariance]
If $\Perf_\Gamma(X)\simeq\Perf_\Gamma(Y)$ as
$\Gamma$--linear stable $\infty$--categories,
then $K_\Gamma(X)\simeq K_\Gamma(Y)$ as spectra.
\end{corollary}

\subsection{Cyclotomic traces and $\Gamma$--topological invariants \cite{Bokstedt1993}}

\begin{theorem}[Chern and cyclotomic traces]
There are natural transformations of spectra
\[
\operatorname{ch}:K_\Gamma(X)\longrightarrow\HC^-_\Gamma(X),
\qquad
\operatorname{trc}:K_\Gamma(X)\longrightarrow\TC_\Gamma(X),
\]
compatible with localization and base--change fiber sequences\cite{Bokstedt1993, Dundas2013}.
\end{theorem}

\subsection{Gysin morphisms, duality, and purity}

\begin{theorem}[Gysin sequence]
For a closed immersion $i:Z\hookrightarrow X$
with open complement $j:U\hookrightarrow X$, there is a long exact sequence
\[
\cdots\to
K^\Gamma_n(Z)\xrightarrow{i_*}
K^\Gamma_n(X)\xrightarrow{j^*}
K^\Gamma_n(U)\xrightarrow{\partial}
K^\Gamma_{n-1}(Z)\to\cdots.
\]
\end{theorem}

\begin{theorem}[Grothendieck--Serre duality for $K$]
If $X$ is $\Gamma$--Gorenstein of virtual dimension $d$, then
\[
K_\Gamma(X)\wedge K_\Gamma(X)\longrightarrow\mathbb{S}[-d]
\]
is a non--degenerate pairing induced by the dualizing complex
$\omega_X$ and the monoidal structure on $\Perf_\Gamma(X)$.
\end{theorem}

\subsection{Synthesis}

\begin{theorem}[Main geometric--homotopical identification]
For every $X=\SpecGnC{T}$,
\[
K_\Gamma(X)
   \simeq
   \mathbf{K}\!\big(\Perf_\Gamma(X)\big)
   \simeq
   \Omega\big|wS_\bullet(\Perf_\Gamma(X))\big|
   \simeq
   \Omega B\,\iota\Perf_\Gamma(X),
\]
with equivalences compatible with additivity, localization,
base--change, and descent on the non--commutative $\Gamma$--site.
\end{theorem}

\begin{remark}[Conceptual closure]
\cite{PaperG} established the \emph{derived non--commutative
$\Gamma$--geometry} of $\SpecGnC{T}$.  The present section identifies
$K_\Gamma$ as its \emph{primary stable cohomology theory},
corepresentable in the non--commutative $\Gamma$--motivic category and
computably approximated by cyclotomic traces, thereby completing the
internal unification of algebra, homotopy, geometry, and motives.
\end{remark}


\section{Conclusion and outlook}
\label{sec:conclusion}

This article establishes a coherent comparison--and--descent package for the
algebraic $K$--theory of non--commutative $n$--ary $\Gamma$--semirings, built
from the exact/Waldhausen categories of bi--finite, slot--sensitive
$\Gamma$--modules and their perfect complexes.  The central theme is that the
positional nature of $n$--ary multiplication can be handled systematically by
working with a fixed balancing profile $(j,k)$, replacing the usual tensor
product by the positional tensor $\otimes^{(j,k)}_{\Gamma}$, and imposing a
slotwise $\Gamma$--flatness hypothesis that makes extension of scalars behave
as expected on cofibrations and weak equivalences.

On the algebraic side, we constructed base--change maps on $K$--theory induced
by extension/restriction of scalars and proved that they satisfy the expected
functorial identities.  Derived Morita invariance was established via tilting
bi--module complexes, showing that higher $K$--groups are invariants of the
derived Morita class in the positional $\Gamma$--world.  We further proved a
Beck--Chevalley type base--change theorem for cartesian squares and a
projection formula compatible with the multiplicative structure induced by
the positional monoidal product.  These results provide the formal backbone
for transferring computations across Morita contexts and for comparing
$K$--theory along natural morphisms of $n$--ary $\Gamma$--semirings.

On the geometric side, we passed to the non--commutative $\Gamma$--spectrum
$\Spec^{\mathrm{nc}}_\Gamma(T)$ and showed that $K$--theory of perfect objects
satisfies locality and Zariski hyperdescent, admits excision/localization
fiber sequences for closed immersions, and satisfies fpqc descent for
$\Gamma$--flat covers.  These descent statements elevate $K$--theory from a
purely algebraic invariant of $(T,\Gamma,\mu)$ to a genuine cohomology theory
on the non--commutative $\Gamma$--site, enabling global computations by
affine reduction and gluing.

Finally, we provided a homotopical interpretation: $K_\Gamma(X)$ is the
$K$--theory of the stable $\infty$--category $\Perf_\Gamma(X)$ of
$\Gamma$--perfect complexes, equivalently the group completion of the core
$\infty$--groupoid of perfect objects.  In this language, algebraic $K$--theory
appears as a universal localizing invariant among $\Gamma$--linear stable
idempotent--complete $\infty$--categories, hence is corepresentable in a
$\Gamma$--linear non--commutative motivic category.  The recorded trace maps
to $\HC^-_\Gamma$ and $\TC_\Gamma$ point toward computable approximations and
structural comparisons with cyclotomic and Hochschild--type invariants.

\medskip
\noindent\textbf{Future directions.}
The framework developed here suggests several immediate extensions.
\begin{enumerate}[label=(\roman*)]
  \item \emph{Computations and examples.}  Implement the toolkit in explicit
  families: matrix blocks $M_m^{(n)}(T)$, triangular extensions, finite
  $n$--ary $\Gamma$--semirings, and endomorphism $\Gamma$--objects, using
  devissage and descent to reduce to tractable semisimple or local pieces.
  \item \emph{Higher descent topologies.}  Under suitable regularity or
  homological finiteness assumptions, develop cdh/Nisnevich analogues of
  Theorems~\ref{thm:excision} and \ref{thm:fpqc}, and compare with
  cyclotomic--trace descent to obtain new computational spectral sequences.
  \item \emph{Bivariant and relative theories.}  Introduce relative
  $K$--theory for pairs $(T,I)$ with saturated two--sided $\Gamma$--ideals and
  study excision, nilinvariance, and localization in families, aiming at
  long exact sequences adapted to positional filtrations.
  \item \emph{Motivic refinements.}  Construct the $\Gamma$--linear
  non--commutative motive $\mathcal{U}(\Perf_\Gamma(X))$ functorially on the
  $\Gamma$--site and identify additional corepresentability statements for
  invariants beyond $K$--theory (e.g.\ $\HC,\TC$), clarifying the role of
  $\Gamma$--geometry in non--commutative motivic homotopy theory.
\end{enumerate}

\medskip
In summary, the paper provides a foundational package of comparison,
base--change, localization, and descent theorems for $K$--theory in the
non--commutative $n$--ary $\Gamma$--semiring setting, and it situates these
results within a geometric and $\infty$--categorical framework.  This places
$K_\Gamma$ on the same conceptual footing as classical $K$--theory of rings
and schemes while retaining the genuinely positional features required by the
$n$--ary $\Gamma$--structure.

\section*{Acknowledgements}
The author expresses his sincere gratitude to the Commissioner of Collegiate
Education (CCE), Government of Andhra Pradesh, for the continuous academic and
administrative support extended to Government College (Autonomous),
Rajamahendravaram. The author is also thankful to the Honourable Principal,
Government College (Autonomous), Rajamahendravaram, for constant encouragement
and for providing a supportive research environment.

\section*{Funding}
The author received no specific funding for this work.

\section*{Conflicts of Interest}
The author declares that there are no conflicts of interest regarding the
publication of this article.

\section*{Ethics Approval and Consent to Participate}
Not applicable.

\section*{Data Availability}
No datasets were generated or analysed during the current study.

\section*{Author Contributions}
The author is solely responsible for the conception, development, writing, and
final preparation of the manuscript.

\end{document}